\documentclass{article}

\usepackage{amsmath}
\usepackage{amssymb}
\usepackage{latexsym}
\usepackage{amsthm}
\usepackage[all]{xy}

\newtheorem{theorem}{Theorem}

\newtheorem{corollary}[theorem]{Corollary}
\newtheorem{lemma}[theorem]{Lemma}

\newcommand{\factor}[2]{{\raise0.7ex\hbox{$#1$} \!\mathord{\left/
 {\vphantom {#1 {#2}}}\right.\kern-\nulldelimiterspace}
\!\lower0.7ex\hbox{${#2}$}}}

\newcommand{\minifactor}[2]{{\raise0.5ex\hbox{\scriptsize$#1$} \!\mathord{\left/
 {\vphantom {#1 {#2}}}\right.\kern-\nulldelimiterspace}
\!\lower0.5ex\hbox{\scriptsize${#2}$}}}

\newcommand{\superminifactor}[2]{{\raise0.3ex\hbox{\tiny$#1$} \!\mathord{\left/
 {\vphantom {#1 {#2}}}\right.\kern-\nulldelimiterspace}
\!\lower0.3ex\hbox{\tiny${#2}$}}}

\newcommand{\comment}[1]{}

\title{The Magnus embedding is a quasi-isometry}
\author{Svetla Vassileva}

\begin{document}

\maketitle

\abstract{We show that the Magnus embedding, which embeds the free solvable group $S_{d,r}$ of rank $r$ and degree $d$ into the wreath product $\mathbb{Z}^r\wr S_{d-1,r}$, is a quasi-isometry.}

\tableofcontents

\section{Introduction}

In 1939 Magnus introduced in \cite{Magnus_original} what has come to be known as the Magnus embedding: an embedding of a group of the type $\factor{F}{N^{\prime}}$ into a matrix group with coefficients in $\mathbb{Z}[\factor{F}{N}]$. In particular, the Magnus embedding is used to embed free solvable groups into certain wreath products. In the 1950s, Fox developed the free differential calculus, which offers a different point of view on the Magnus embedding. This was the beginning of the study of algorithmic problems in free solvable groups, but it also opened the door to extremely interesting geometric questions. For example, in \cite{MRUV} Miasnikov et al. use Fox calculus and ideas of Droms, Servatius and Lewin \cite{DSL} to compute the Magnus embedding in uniform polynomial time and show that the bounded geodesic length problem is NP-complete. In \cite{my_baby}, we use the Magnus embedding to prove that the conjugacy and conjugacy search problems are solvable in polynomial time. In this paper, we expand on these results and show that the Magnus embedding is a quasi-isometry.  

Solvable groups have long been studied from the view-point of combinatorial group theory. However, they have recently been at the focus of  intense research in geometric group theory. One can trace back this interest to Gromov's result \cite{Gromov} on groups with polynomial growth. He showed that a group quasi-isometric to nilpotent group is virtually nilpotent. Erschler showed in \cite{Erschler} that this is not the case for solvable groups in general, that is, there are groups which are quasi-isometric to a solvable group, but are not virtually solvable. This gave rise to an intense study of finitely generated groups up to quasi-isometry. There have been some positive results for certain classes of solvable groups. For example, Farb and Mosher, \cite{FarbMosher1}, \cite{FarbMosher2}, showed that solvable Baumslag-Solitar groups are quasi-isometrically rigid. 

Another view-point from which the present result is interesting comes from the study of distortion of subgroups in wreath products. The Magnus embedding gives an embedding of free metabelian groups in the wreath product of free abelian groups. In \cite{DavisDistortion}, Davis and Olshanskii show that every finitely generated subgroup of $A\wr \mathbb{Z}$, where $A$ is abelian, has polynomial distortion. Further, they observe that the distortion in $\mathbb{Z}^r\wr \mathbb{Z}^r$ gives a upper bound to the distortion in free metabelian groups. A corollary of our result is that the distortion of a subgroup $H$ in the free metabelian group $M_r$ of rank $r$ is in fact the same as the distortion of $H$ in $\mathbb{Z}^r \wr \mathbb{Z}^r$.

\section{Geodesic length in $A\wr B$ and $F/N^{\prime}$}

Let $F = \langle x_1, \ldots, x_r \rangle$ be the free group of rank $r$ and $N$ a normal subgroup of $F$. Denote by $N^{\prime}$ the commutator subgroup $$N^{\prime} = [N,N] = \langle [x,y] \mid x,y \in N \rangle.$$ Let $\bar{} : F \rightarrow F/N$ and $\mu: F \rightarrow \factor{F}{N^{\prime}}$ be the canonical epimorphisms. Since $N^{\prime} \leq N$, the diagram 

\begin{displaymath}
\xymatrix{
	F \ar[r]^{\mu} \ar[d]_{\bar{\phantom{.}}} & \factor{F}{N^{\prime}} \ar[ld] \\
	\factor{F}{N}  
}
\end{displaymath}
commutes and we may use the symbol $x_i$ to denote any of $x_i$, $\overline{x_i}$ or $\mu(x_i)$, depending on the context. 

Let $A$ be a free abelian group of rank $r$, generated by $\{a_1, \ldots, a_r\}$ and denote by $B$ the group $\factor{F}{N}$. 
The homomorphism $$\phi: \factor{F}{N^{\prime}} \hookrightarrow A\wr B$$ given by $$\phi(g) = \overline{g} \cdot a_1^{\overline{\minifactor{\partial g}{\partial x_1}}} \ldots a_r^{\overline{\minifactor{\partial g}{\partial x_r}}},$$
where $\factor{\partial g}{\partial x_i}$ is the Fox derivative of $g$ with respect to the generator $x_i$, is injective \cite{Magnus_original}, and is called the \emph{Magnus embedding}. For a good description, see \cite{HNbook}.

The Cayley graph, $\mathrm{Cay}(G, \{x_1, \ldots, x_r\})$, of a group $G$ with respect to the generating set $\{x_1, \ldots, x_r\}$ has as vertex set the elements of $G$. Two vertices $g,h\in \mathrm{Cay}(G)$ are connected by a directed edge $(g,h)$ if for some generator $x_i$ of $G$, $h = gx_i$. We assume that if $x_i$ is in the generating set of $G$, then the formal inverse $x_i^{-1}$ is not. Instead, we allow paths to traverse an edge in either direction.

We denote by $\| g \|_G$ the length of the shortest word which is equal to the element $g$ in $G$. Hence $\|g\|_G$ is the distance from $1_G$ to $g$ in $\mathrm{Cay}(G)$, so we call it the geodesic length.

\subsection*{Geodesic length in $A\wr B$}

For two groups $A$ and $B$, let $A^{(B)}$ denote the set of functions of finite support from $B$ to $A$, which forms a group under point-wise multiplication. Then the \emph{wreath product} $A\wr B$ of $A$ and $B$ is defined as the semi-direct product $B\ltimes A^{(B)}$. The action of $b\in B$ on a function $f\in A^{(B)}$ is defined by $f^b(x) = f(xb^{-1})$. 

Regarding $a\in A$ as the function 
\[
a(b) = \left\{ \begin{array}{l} a \mbox{  if $b=1$,} \\ 1 \mbox{  otherwise,}\end{array}\right.
\]
the groups $A$ and $B$ canonically embed in $A \wr B$.  Further, $A \wr B$ is 
generated by the union of the generating sets of $A$ and $B$ and the relation 
$[a^{b}, (a^{\prime})^{b^{\prime}}]=1$ holds for all $a, a^{\prime}\in A$ and $b,b^{\prime} \in B$. Moreover, if $w$ is a word given in the generators of $A\wr B$, it can be rewritten in the ``standard form'' $$w = b A_1^{B_1}\hdots A_{k-1}^{B_{k -1}} A_k^{B_k}, $$ where $b\in B$, $A_1, \ldots, A_k$ are non-trivial elements of $A$, $B_1, \ldots, B_{k}\in B$ and $B_i \neq B_j$ for $i\neq j$. 

Walter Parry \cite{Parry} showed that the geodesic length of $w$ is given by $$\|w\|_{A\wr B} = \| b \|_B + \sum\limits_{i=1}^{k} \|A_i\|_A + \mathcal{L}_{\mathrm{Cay}(B)}(B_1, \ldots, B_{k}), $$ where $\mathcal{L}_{\mathrm{Cay(B)}}(B_1, \ldots, B_{k})$ denotes the length of the shortest circuit in $\mathrm{Cay}(B)$ passing through the vertices in $V =\{1_B, B_1, \ldots, B_k\}$. We will call this a minimum length cycle.

\subsection*{Geodesic length in $F/N^{\prime}$}

Miasnikov et. al \cite{MRUV} proved that the problem of finding geodesics in free metabelian groups is NP-hard. In doing so, they gave a description of geodesics in $\factor{F}{N^{\prime}}$. A word $w$ in generators of $F$ defines the path $p_w$ that $w$ traces in $\mathrm{Cay}(B)$. One can define a corresponding flow function $$\pi_w : E\big(\mathrm{Cay}(B)\big) \rightarrow \mathbb{Z}, $$ where $\pi_w(e)$ counts how many times $p_w$ passes through the edge $e$. Here, whenever $p_w$ traverses an edge $e$ in the negative direction, we say that $p_w$ passes through this edge $-1$ times. This gives rise to a labeled subgraph $\Gamma$ of $\mathrm{Cay}(B)$ consisting of the edges $$E(\Gamma) = \{e\in E\big(\mathrm{Cay}(B)\big) \mid \pi_w(e)\neq 0\} $$ and with vertices the endpoints of these edges. The edges in $E(\Gamma)$ are labelled by their corresponding flow $\pi_w(e)$. 

A remarkable fact proved in \cite{MRUV} is that two words $u$ and $v$ in $F$ define the same flow function in $\mathrm{Cay}(\factor{F}{N})$ if and only if they are equal as elements of $\factor{F}{N^{\prime}}$. Therefore we can study elements in $\factor{F}{N^{\prime}}$ through their flows. 

Denote by $C_1, \ldots, C_l$ the connected components of $\Gamma$ and let $Q$ be a minimal forest connecting them. That is, $Q$ is a subgraph of $\mathrm{Cay}(B)$ such that $\Delta = Q\cup C_1\cup \hdots \cup C_l$ is connected and the edge set $E(Q)$ is minimal. 

Define $\Delta^*$ as follows. It has the same vertex set as $\Delta$. For each edge $(u,v)$ in $C_1 \cup \ldots \cup C_l$  include $|\pi_w(u,v)|$ undirected edges between $u$ and $v$ and for each edge $(u,v)$ in $Q$ add two undirected edges between $u$ and $v$. To every path in $\Delta^*$ we can associate a path in $\Delta$ (and hence in $\mathrm{Cay}(B)$) in the obvious way. 

It is proved in \cite{MRUV} that the label in $\mathrm{Cay}(B)$ of the path corresponding to an Euler tour of the edges in $\Delta^*$ starting at the identity gives a geodesic for $w$ as element of $\factor{F}{N^{\prime}}$. Hence, the geodesic length of $w$ in $\factor{F}{N^{\prime}}$ is equal to the number of edges in $\Delta^*$, i.e.,  
\begin{equation}
\label{geodesicFNprime}
\|w\|_{\minifactor{F}{N^{\prime}}} = \sum\limits_{e\in \mathrm{supp}(p_w)} |\pi_w(e)| + 2|E(Q)|. 
\end{equation}

\comment{
It is proved in \cite{MRUV} that any word labelling a path in $\Delta$ of minimal length starting at the identity and including every edge of $\Delta$ is a geodesic for $w$ as an element of $\factor{F}{N^{\prime}}$. Hence, the geodesic length of $w$ in $\factor{F}{N^{\prime}}$ is given by 
\begin{equation}
\label{geodesicFNprime}
\|w\|_{\minifactor{F}{N^{\prime}}} = \sum\limits_{e\in \mathrm{supp}(p_w)} |\pi_w(e)| + 2|E(Q)|. 
\end{equation}
}

\section{The Magnus embedding is a quasi-isometry}

We show that $\phi: \factor{F}{N^{\prime}} \hookrightarrow A\wr B$ is a quasi-isometry. More precisely, the following theorem is true. 

\begin{theorem}
Let $w$ be a word in $\factor{F}{N^{\prime}}$ given as a product of generators $x_1, \ldots, x_r$. Then $$\frac{1}{2(r+1)} \|w\|_{\minifactor{F}{N^{\prime}}} \leq \|\phi(w)\|_{A\wr B} \leq 3 \|w\|_{\minifactor{F}{N^{\prime}}}.$$
\end{theorem}

\begin{proof}

Recall that $$\phi(w) = \overline{w} \cdot a_1^{\overline{\minifactor{\partial w}{\partial x_1}}} \hdots a_r^{\overline{\minifactor{\partial w}{\partial x_r}}}. $$ 
Since each $\overline{\minifactor{\partial w}{\partial x_i}}$ is in the group ring $\mathbb{Z}[B]$, we can write for each $i = 1, \ldots, r$, 
\begin{equation}
\label{expression dw}
\overline{\frac{\partial w}{\partial x_i}} = \sum_{b\in B} u_b^{(i)} b, 
\end{equation}
with $u_b^{(i)}\in \mathbb{Z}$ for all $b\in B$ and for all $i=1, \ldots, r$. Thus, we can rewrite $\phi(w)$ as $$\phi(w) = \overline{w} \cdot \prod_{b\in B}\big(a_1^{u_b^{(1)}}\big)^b \hdots \prod_{b\in B}\big(a_r^{u_b^{(r)}}\big)^b. $$ For every $b\in B$, let 
\begin{equation}
\label{defAi}
A_b = \prod_i a_i^{u_b^{(i)}}.
\end{equation} 
Using the commutator relations of the wreath product, we can rewrite $\phi(w)$ in standard form: 
\begin{equation}
\label{expression phi(w)}
\phi(w) = \overline{w} \cdot \prod_{b\in B} A_b^b. 
\end{equation}
Since, by the definition of the Fox derivatives, the sum (\ref{expression dw}) is finite, the product (\ref{expression phi(w)}) is finite as well.  Hence we can write $$\phi(w) = \overline{w}\cdot A_1^{B_1}\hdots A_k^{B_k}, $$
where $A_1, \ldots, A_k \in A$ are non-trivial and all $B_1,\ldots, B_k \in B$ are distinct. 

\begin{lemma}
\label{sumA_i}
$$\sum_i \| A_i \|_A = \sum_{e\in \mathrm{supp}(p_w)} |\pi_w(e)|$$
\end{lemma}

\begin{proof}
Miasnikov et al. prove that for every generator $x_i$, 
\begin{equation}
\label{Fox=flow}
\overline{\frac{\partial w}{\partial x_i}} = \sum_{b\in B} \pi_w(b,bx_i) b,  
\end{equation}
i.e., that the coefficient of $b$ in the expression for the Fox derivative of $w$ with respect to $x_i$ is equal to the flow through the edge $(b,bx_i)$. It follows from equation (\ref{defAi}) that $\sum_i \|A_i\|_A$ is the sum of absolute values of the coefficients in all the Fox derivatives. Indeed, 

\begin{eqnarray*}
\sum_i \| A_i \|_A &=& \sum_i \| A_{B_i} \|_A = \sum_i \|\prod_j a_j^{u_{B_i}^{(j)}}\| = \sum_i \sum_j \big|u_{B_i}^{(j)}\big| \\
&=& \sum_i \big|\pi_w(B_i, B_ix_i)\big| = \sum_{e\in \mathrm{supp}(p_w)} |\pi_w(e)|
\end{eqnarray*}
\end{proof}

\begin{lemma}
$$\|\phi(w)\|_{A\wr B} \leq 3 \|w\|_{\minifactor{F}{N^{\prime}}}.$$
\end{lemma}

\begin{proof}

Note that since the elements $B_1, \ldots, B_k$ come from the non-zero terms in (\ref{expression dw}), it follows from (\ref{Fox=flow}) that all the vertices in $V$ appear in the vertex set of $\Gamma$. So $V \subseteq V(\Gamma)$. As described above, the geodesic of $w$ in $\factor{F}{N^{\prime}}$ is found by touring all connected components of $\Gamma$ and connecting them in an optimal way. A geodesic representation of $w$ in $\factor{F}{N^{\prime}}$ corresponds to a tour of the vertices in $V(\Gamma)$ and therefore to a tour of the vertices in $V$. Clearly, it is longer than a minimal length tour on $V$, i.e., $$\|w\|_{\minifactor{F}{N^{\prime}}} \geq \mathcal{L}(V).$$
Moreover, since $B=\factor{F}{N}$ is a quotient of $\factor{F}{N^{\prime}}$, it is easy to see that $\|\overline{w}\|_{B} \leq \|w\|_{\minifactor{F}{N^{\prime}}}$. 
Further, since $\|w\|_{\minifactor{F}{N^{\prime}}}$ is given as in (\ref{geodesicFNprime}), it follows 
from Lemma~\ref{sumA_i} that $\sum_i \|A_i\|_A \leq \|w\|_{\minifactor{F}{N^{\prime}}}$. Therefore,  
\begin{eqnarray*}
\|\phi(w)\|_{A\wr B} &=& \|w\|_B + \sum_i \|A_i\|_A + \big|\mathcal{L}\big(V\big)\big| \\
&\leq& \|w\|_{\minifactor{F}{N^{\prime}}} + \|w\|_{\minifactor{F}{N^{\prime}}} + \|w\|_{\minifactor{F}{N^{\prime}}} = 3\|w\|_{\minifactor{F}{N^{\prime}}}.
\end{eqnarray*}
\end{proof}

Let $\mathcal{T}$ be a shortest circuit of the vertices in $V$. From $\mathcal{T}$ we will produce a tour of the vertices in $\Gamma$. We will consider how the two vertex sets differ.

\comment{This is a little bit trickier, as $V$ is smaller than $V(\Gamma)$ and so, a priori, a minimal length cycle on the former gets us no nearer to a minimal forest for the connected components of the latter. So let us think about how the two vertex sets differ. Let us go back to our key formula: $$\overline{\frac{\partial w}{\partial x_i}} = \sum_{b\in B} \pi_w(b,bx_i) b.$$
}
The vertices in $V(\Gamma)$ consist of endpoints of all edges $e$ with $\pi_w(e)\neq 0$. The vertices in $V$ consist of only the initial endpoints of these edges, since $$\overline{\frac{\partial w}{\partial x_i}} = \sum_{b\in B} \pi_w(b,bx_i) b.$$

\comment{
we only have vertices whose corresponding elements appear in the expression of some Fox derivative, i.e., $b\in V$ if and only if $\pi_w(b,bx_i)\neq 0$ for some generator $x_i$. In other words, $v\in V$ if and only if $v$ is the \emph{initial} vertex of some edge $e$ with $\pi_w(e)\neq 0$. On the other hand,  So $$|V(\Gamma)| \leq r|V|.$$ 
}

Moreover, the vertices which are in $V(\Gamma) \setminus V$ are distance 1 away from a vertex in $V$ (because they correspond to the terminal vertex of edge whose initial vertex is in $V$). Now one can produce a tour of the vertices in $\Gamma$ by following $\mathcal{T}$ and simply `hopping' (at an extra cost of 2) to visit a `missing' vertex. Since for every vertex in $V$ there are at most $r$ `missing' vertices (one for each generator), one obtains a tour of the vertices of $\Gamma$ which has length at most $\mathcal{T} + 2r\mathcal{T} = (2r+1)\mathcal{T}$. 

Since $\mathcal{T}$ has minimal length $\mathcal{L}(V)$, then we obtain a tour of $V(\Gamma)$ of length at most $(2r+1) \mathcal{L}\big(V\big)$. 

It is clear that touring the connected components of $\Gamma$ and connecting them in a minimal way will give a minimal tour of the vertices in $V(\Gamma)$ and so its length will be less than that of $\mathcal{T}$. In other words, 

$$\|w\|_{\minifactor{F}{N^{\prime}}} \leq (2r+1) \mathcal{L}\big(V\big) \leq (2r+1) \|\phi(w)\|_{A\wr B} .$$

\comment{
\begin{eqnarray*}
\|w\|_{\minifactor{F}{N^{\prime}}} &\leq& (2r+1) \big| \mathcal{L}\big(V\big) \big| + \sum_{e \in \mathrm{supp}(p_w)} |\pi_w(e)|\\ 
&=& (2r+1) \big| \mathcal{L}\big(V\big) \big| + \sum_i \|A_i\|_A\\
&\leq& (2r+1) \|\phi(w)\|_{A\wr B} + \|\phi(w)\|_{A\wr B} \\
&\leq& 2(r+1)\|\phi(w)\|_{A\wr B}.
\end{eqnarray*}
} 
This finishes the proof of the theorem. 

\end{proof}

\begin{corollary}
The Magnus embedding of the free solvable group $S_{d,r}$ of class $d$ and rank $r$ into the wreath product $\mathbb{Z}^r\wr S_{d-1, r}$ is a quasi-isometry.
\end{corollary}

\bibliographystyle{alpha}
\bibliography{MagnusEmbeddingQI}

\end{document}